\documentclass[11pt]{amsart}%amsart

\usepackage{amssymb}
\usepackage{amsthm}
\usepackage{mathtools}
\usepackage{color}
\usepackage{amsmath}
\usepackage{graphicx}
\usepackage[T1]{fontenc}
\usepackage[english]{babel}
\usepackage{enumerate}
\usepackage{babel}
\usepackage[left=3cm,right=3cm,top=3cm,bottom=3cm]{geometry}
\usepackage[pdftex,breaklinks,colorlinks,citecolor=blue,urlcolor=blue]{hyperref}
\usepackage{verbatim}
\numberwithin{equation}{section}

\newtheorem{theo}{Theorem}[section] 
\newtheorem{lem}[theo]{Lemma}

\newtheorem{rem}[theo]{Remark}

\newtheorem{prop}[theo]{Proposition}

\newcommand{\R}{\mathbb{R}}
\newcommand{\C}{\mathbb{C}}

\def \intb {\int_{\partial \Omega}}
\def \intOm {\int_{\Omega}}
\def \nn   {\vec{n}}

\def \WW {\mathcal{W}}
\newcommand{\nablaT}{ {\slash\!\!\!\!  \nabla} }
\def \Nls {{\rm{(NLS$_\Omega$) }}}
\def \NNls {{\rm{NLS$_\Omega$ }}}
\def \nls {{\rm{(NLS) }}}
\def \nnls {{\rm{NLS }}}

\DeclareMathOperator{\re}{Re}
\DeclareMathOperator{\im}{Im}

\title[Blow-up on NLS$_{\Omega}$]{On Blow-up solutions to the nonlinear Schr\"odinger equation in the exterior of a convex obstacle}
\author[Oussama Landoulsi]{Oussama Landoulsi$^1$}
\email[Oussama Landoulsi]{landoulsi@math.univ-paris13.fr}
\thanks{$^1$LAGA, UMR 7539, Institut Galil\'ee, Universit\'e Sorbonne Paris Nord}
\address{LAGA, UMR 7539, Institut Galil\'ee, Universit\'e Sorbonne Paris Nord}

\date{\today}
\makeatletter
\@namedef{subjclassname@2010}{Mathematics Subject Classification}
\makeatother


\begin{document}
\maketitle

\begin{abstract}
%In this paper, we prove the existence of blow-up solutions for negative energy using a classical convexity argument for the focusing nonlinear Schr\"odinger equation in the exterior of the unit ball on $\R^d$ with Dirichlet boundary conditions. %Moreover, we give an explicit solution $u$ which blows up in finite time for $L^2$-critical case, i.e $p=1+\frac{4}{d}$. 
%Moreover, we study the behavior of the solutions under the mass-energy threshold, in particular, we prove that the blow-up criteria for the solution with symmetric initial data is the same as for the problem posed on Euclidean space given by previous work of T. Duyckaerts, S.Roudenko and J.Holmer.  \\ 
 In this paper, we consider the Schr\"odinger equation with a mass-supercritical focusing nonlinearity, in the exterior of a smooth, compact, convex obstacle of $\R^d$ with Dirichlet boundary conditions. We prove that solutions with negative energy blow up in finite time. Assuming furthermore that the nonlinearity is energy-subcritical, we also prove (under additional symmetry conditions) blow-up with the same optimal ground-state criterion than in the work of Holmer and Roudenko on $\R^d$. The classical proof of Glassey, based on the concavity of the variance, fails in the exterior of an obstacle because of the appearance of boundary terms with an unfavorable sign in the second derivative of the variance. The main idea of our proof is to introduce a new modified variance which is bounded from below and strictly concave for the solutions that we consider.
\end{abstract}

\tableofcontents
\newpage
 \section{Introduction}
We consider the focusing nonlinear Schr\"odinger equation with Dirichlet~boundary~conditions.
	\begin{multline}
	\label{NLS} 
		%(\text{NLS}{_\Omega})
		\tag{NLS$_{\Omega}$}
			\left\{
			\begin{array}{rrrrrrrr}
			\begin{aligned}
	i\partial_tu+\Delta_{\Omega} u &= -|u|^{p-1}u  \qquad  & (t,x)\in \R \times\Omega ,\\ u(t_0,x) &=u_0(x)   & x \in \Omega  , \\
	u(t,x)&=0  &(t,x)\in \R \times \partial\Omega ,
		\end{aligned}
			\end{array}
		\right. \end{multline}  
where $\Omega=\R^d \setminus \Theta$ is the exterior of a smooth, compact and convex obstacle on $\R^d,$ 
$\Delta_{\Omega}$ is the Dirichlet Laplace operator, $\partial_t$ is the derivative with respect to the time variable and $t_0 \in \R$ is the initial time. The function $u$ is complex-valued, $u:\R \times \Omega \longrightarrow \C, \, (t,x)\longmapsto u(t,x).$ \\ 
 
%The local well-posedness for the \Nls equation in the exterior of a smooth, compact, convex domain was studied in several articles and it is now well understood in many cases. Local existence and uniqueness are usually proved by contraction mapping methods via Strichartz estimates. In \cite{Ivanovici10}, O.\,Ivanovici proved the full range of Strichartz estimates for \Nls except the end point case, using the Melrose and Taylor parametrix, see also  \cite{BlairSmithSogge2012}, \cite{BuGeTz04a}, \cite{killip2015riesz}. The local well-posedness in $H^1_0(\Omega),$ for $1<p<5$ in dimension $d=3$, for \Nls equation in the exterior of a non-trapping domain was obtained by L.\,Vega and F.\,Planchon in \cite{PlVe09}. F.\,Planchon and O.\,Ivanovici \cite{MR2683754} extended the result to the quintic Schr\"odinger equation outside a non-trapping domain, see also \cite{Farah15} for $d=2,$ \cite{MR3483844,Kai17} for global existence for $d=3$ and \cite{KiVisZha16} for the defocusing case %see also  \cite{KiVisZha16} for the defocusing case.   \cite{killip2015riesz},. \\ 
%(Cf. Proposition \ref{well-posed} below for a precise local well posedness statement needed for our purpose). \\ %\cite{Ou19}
 
The local well-posedness for the \NNls equation in the exterior of a smooth, compact and convex domain was studied in several articles and is now well understood in many cases. Local existence and uniqueness are usually proved by contraction mapping methods via Strichartz estimates. The Cauchy theory for the \NNls equation with initial data in $H^1_0(\Omega),$ was initiated in  \cite{BuGeTz04a}. The authors proved 
first a Strichartz estimates under some restriction, which led to a local existence result for $p<3, d=3.$ This results was later extended for all $1<p\leq 5,$ in dimension $d=3$, in particular, see \cite{AnRA08} for the cubic nonlinearity, \cite{PlVe09} for $1<p<5$ and \cite{OanaFabrice2010} for $p=5,$ see also \cite{BlairSmithSogge2012}. The full range of Strichartz estimates for \NNls except the end point was obtained in  \cite{Ivanovici10}. Therefore, the \NNls equation is locally-well posed in $H^1_0(\Omega),$ for $1<p <\frac{d+2}{d-2}, d\geq 4 $ and $1<p<\infty,$ for $d=2$ (Cf. Proposition \ref{well-posed} below for a precise local well posedness statement needed for our purpose).  \\

%In \cite{Ivanovici10}, O.\,Ivanovici proved the full range of Strichartz estimates for \NNls except the end point case, see also \cite{BlairSmithSogge2012},  \cite{AnRa08}. The local well-posedness in $H^1_0(\Omega),$ in particular, was obtained in \cite{PlVe09} for $1<p<5$ and in \cite{OanaFrabrice2010}, for $p=5$ in dimension $d=3$, see also \cite{Farah15} for $d=2$ and \cite{KiVisnaZhang16}, \cite{Kai17} for $d=3$ and \cite{KiVisZha16} for the defocusing case (Cf. Proposition \ref{well-posed} below for a precise local well posedness statement needed for our purpose). \\ 

 The solutions of (NLS$_{\Omega}$) satisfy the mass and energy conservation laws: 
 \begin{align*}
 M[u(t)] &:= \int_{\Omega} |u(t,x)|^2 dx = M[u_0] , \\
E[u(t)] &:= \frac{1}{2} \int_{\Omega} |\nabla u(t,x)|^2 dx - \frac{1}{p+1} \int_{\Omega}\left|u(t,x)\right|^{p+1}dx=E[u_0]. 
\end{align*}

In~\cite{KiVisnaZhang16}, the authors proved global existence and scattering of solutions for the focusing $3$d cubic \NNls equation, whenever the initial data satisfies a smallness criterion given by the ground state threshold, see also \cite{Kai17} for $\frac{7}{3}<p<5,d=3$. The criterion is expressed in terms of the scale-invariant quantities $\left\|u_0\right\|_{L^2}\left\| \nabla u_0 \right\|_{L^2}$ and $M[u]E[u].$ Moreover, in  \cite{XuZhaozhengrevisedThresScattering2018} the authors revisited the proof of scattering using Dodson and Murphy's approach \cite{DodsonMurphy17}, \cite{DodsonMurphy18} and the dispersive estimate established in \cite{IvanoLebo17}. In \cite{Ou19}, we constructed a solitary wave solution for \Nls behaving asymptotically as a soliton on $\R^3$, as large time. This solution is global, does not scatter and prove the optimality of the threshold for scattering given above. \\

All results mentioned above for \Nls  concern global solutions but the existence of blow-up solutions is still an open question which is the purpose of this paper. \\ 

Before stating  our blow-up results for the \NNls equation, let us recall the proof of the classical blow-up criterion of Vlasov-Petrishev-Talanov \cite{Vlasov1970}, Zakharov \cite{Zakharov1972} and Glassey \cite{Glassey1977} which states that finite variance, negative energy solutions break down in finite time. This proof is a convexity argument on the variance $V(t)$ defined as the following, 
\begin{equation*}
   V(t):= V(u(t))= \int_{\R^d} |x|^2 |u(t,x)|^2 dx.
\end{equation*}
Assuming $V(0)< \infty$, the following virial identity holds:

\begin{equation*}
   \frac{1}{16}\frac{d^2}{dt^2}V(t)=E_{\footnotesize{\R^d }}(u)-\frac{1}{2}\left(\frac{d}{2}-\frac{d+2}{p+1} \right)\left\| u \right\|_{L^{p+1}\footnotesize{\left(\R^d\right)}}^{p+1},
   \end{equation*}
where $E_{\footnotesize{\R^d }}(u)=\frac{1}{2} \left\|\nabla u \right\|_{L^2\footnotesize{\left(\R^d\right)}}-\frac{1}{p+1}\left\| u \right\|_{L^{p+1}\footnotesize{\left(\R^d\right)}}^{p+1}.$ If $p>1+\frac{d}{4}$ and $E_{\footnotesize{\R^d }}(u) < 0 $ then $u$ blows up in finite time. As proved in \cite{HoRo07}, in the energy subcritical case ($d\leq 2$ or $p<1+\frac{4}{d-2}$), the assumption $E_{\R^d}(u)<0$ can be weakened to a condition on the initial data which can be formulated in term of the ground state (see Theorem \ref{theo threshold on d=2} below).\\ 

This proof does not adapt directly to the case of an exterior domain because the boundary term in the virial identity does not have a favorable sign, 
\begin{equation*}
  \frac{1}{16}  \frac{d^2}{dt^2} V(u(t))= E [u]- \frac{1}{2} \left( \frac{d}{2}  - \frac{d+2}{p+1} \right)\intOm |u|^{p+1} \, dx - \frac{1}{4 } \intb \left| \nabla u \right|^2 \, (x.\nn) \, d\sigma(x),
\end{equation*}
where $\nn$ is the unit outward normal vector. One can see that $x.\nn \leq 0 $ on $\partial \Omega.$ In this work, we will define a new shifted variance $\mathcal{V}(t)$ which will allow us to control the boundary term and to prove the existence of blow-up solution for the \NNls equation. In the energy subcritical case, with an additional symmetry assumption on the initial data, we will prove blow-up with the sufficient condition obtained in \cite{HoRo07} on the Euclidean space. \\

%We will also prove by a similar argument that the threshold for blow-up for symmetric solution of \Nls is the same as for the problem posed on whole Euclidean space $\R^d$. \\ 

%Moreover, the {\rm(NLS)} equation posed on the whole Euclidean space $\R^d$ is invariant by the scaling transformation, that is, 
 %\begin{equation*}
% u(t,x) \longmapsto \lambda^{\frac{2}{p-1}}u(\lambda x,\lambda^2 t) \;, \; \text{ for }  \lambda >0.
 %\end{equation*}
 
%Since the presence of the obstacle does not change the intrinsic dimensionality of the problem, thus, this scaling identifies the critical Sobolev space $\dot{H}^{s_c}_{x}$ for \Nls as well, where the critical regularity $s_c$ is given by $s_c:= \frac{d}{2} - \frac{2}{p-1}$. The case $ s_c =0$ is referred to as mass-critical or $L^2$-critical. \\ %  and the case $s_c = 1$ is called energy-critical or $\dot{H}^1$-critical. \\  
%We denote, $B(0,1):=\{x \in \R^d \; \backslash\;  |x|\leq 1 \},$ where  $\left| \cdot \right|$ is the Euclidean-norm with $d\geq 1$.  \\

Next, we recall the needed local well-posedness property for the \NNls equation posed outside a convex obstacle.  %Recall that, $(r,q)$  is Strichartz admission pair if, 
%\begin{equation}
%\label{admiss}
%\frac{2}{q}+\frac{d}{r}=\frac{d}{2}, \text{ where } \; q>2  \; \text{ and } \;  2 \leq r \leq \infty. 
%\end{equation}
\begin{prop}
\label{well-posed}
Assume $p >1$ if $d = 2$ and $ 1 < p <\frac{d+2}{d-2}$ if $d\geq3.$  
Let $u_0 \in  H^1_0(\Omega)$ then there exists $ T >0$ and a unique solution $u(t)$ of \Nls equation with $u \in C([- T,T], H^1_0(\Omega))$\\ %\cap L^{q}([- T,T],L^r(\Omega)),$ where $(q,r)$ satisfies \eqref{admiss}. \\
Assume $d=3$ and $p>2 $. Let $u_0 \in H^2\cap H^1_0(\Omega)$ then there exists $T>0$ and a unique solution~$u(t)$ of \Nls equation with $u \in C([- T,T], H^2 \cap H^1_0(\Omega)).$     
\end{prop}
We omit the standard proof of Proposition \ref{well-posed}. The local existence and uniqueness in $H^1_0(\Omega)$ can be carried out by classical methods, using fixed point argument via Strichartz estimates. The proof is very similar to the one for the \nnls equation posed on the whole Euclidean space.
Moreover, the local existence of solutions for the \NNls equation in $H^2\cap H^1_0(\Omega)$ can be established using the fact that $H^2$ is an algebra and the following continuous embedding for any smooth domain $\Omega \subset \R^3$, $H^2(\Omega) \subset L^{\infty}(\Omega),$  see \cite[Proposition 2.1]{BuGeTz04a}. Thus we don't have to control the  nonlinearity growth but we just need regularity of the nonlinear term.  \\

It is classical that the solution $u$ can be extended to a maximal time interval $I= (-T_{-},T_{+})$ of existence. If $T_{+}=+\infty$ (respectively $T_{-}=-\infty $) then the solution is global for positive time (respectively for negative time) and if 
$T_{+}< \infty$ (respectively $|T_{-}|<\infty) $ then the solution blows up in finite time and 
%Either $T_+= \infty$ (respectively $T_{-}=\infty $) or $T_+< \infty$ (respectively $T_{-}<\infty) $ and 
$$\lim\limits_{t \to T_{+} } \left\|u(t, \cdot) \right\|_{H^1_0(\Omega)}=+\infty , \quad \text{  respectively } \lim\limits_{t \to  T_{-}} \left\|u(t, \cdot) \right\|_{H^1_0(\Omega)}=+\infty. $$

Let $\mathcal{H}(\Omega) \subset H^1_0(\Omega)$ be a space where a Cauchy theory for \Nls is available, in particular, from Proposition \ref{well-posed}, for $d=2$  one can consider $\mathcal{H}(\Omega)=H^1_0(\Omega)$ and for $d=3,$ and $p>2,$ it suffices to take $\mathcal{H}(\Omega)=H^2\cap H^1_0(\Omega).$ \\ 

\newpage

Now we state our main results.  
\begin{theo}
\label{theo outside spehre}
Assume $\Theta=B(0,R)$ and $p\geq 5.$ 
\begin{itemize}
    \item for $d=2,$ let $u_0 \in H^1_0(\Omega)$  such that $ E[u_0]+\frac{1}{8R^2}M[u_0]<0$ and $|x|u_0 \in L^2(\Omega),$  
    \item for $d \geq 3$, let $u_0 \in \mathcal{H}(\Omega) $ such that $\; E[u_0]<0$ and $|x|u_0 \in L^2(\Omega),$  
\end{itemize}
and let $u$ be the corresponding solution of {\rm{(NLS$_\Omega$)}}. Then the solution $u$ blows up in finite time. %with maximal time interval $I$ of existence then the length of $I$ is finite and the solution $u$ blows up in finite time.
\end{theo}
Next, we extend the above result to any smooth, compact and convex obstacle $\Theta,$ such that the following holds in dimension $d\geq 2:$ Let  $M:= \displaystyle  \max_{x \in \partial \Theta} \left(|x|\right)$ and $m:= \displaystyle \min_{x\in \partial \Theta}\left( |x|  \right),$ then 
\begin{equation} 
\label{obs-type}
\frac{M}{m} < \frac{d}{d-1}.
 \end{equation}
 
\begin{theo}
\label{theo-convex}
Assume that $\Theta$ satisfies \eqref{obs-type} and $p\geq 1+\frac{4}{d-\frac{M}{m}(d-1)}. $
\begin{itemize}
\item for $d=2,$ let $u_0 \in H^1_0(\Omega)$ such that $E[u_0]+\frac{M}{8 m^3} M[u_0]<0 $ and $|x|u_0 \in L^2(\Omega),$ 
\item for $d\geq 3,$ let $u_0 \in  \mathcal{H}(\Omega)$ such that $E[u_0]<0$ and 
$|x| u_0 \in L^2(\Omega),$
\end{itemize}
and let $u$ be the corresponding solution of {\rm{(NLS$_\Omega$)}}. Then the solution $u$ blows up in finite time. 
%with maximal time interval $I$ of existence then the length of $I$ is finite and the solution $u$ blows up in finite time. 
\end{theo}

\begin{rem}
Let us mention that, if $\Theta$ and $p$ satisfy the assumptions of the Theorems  and $u_0\in \mathcal{H}(\Omega)\setminus\{0\}$, then the solution with initial data $\lambda u_0$ blows up in finite time for large $\lambda$. 
\end{rem}

In the following result, we consider $\Theta$ to be a smooth, compact and convex obstacle which is invariant with respect to the transformations $x_j\longmapsto -x_j ,$ for $j=1, \ldots,d,$ that is,
 if $ x=(x_j)_{1\leq j \leq d}\in \Theta,$ then $(x_1,\cdot\cdot\cdot, -x_j, ,\cdot\cdot\cdot , x_d)  \in \Theta  .$ For example, $\Theta$ might a ball or the volume delimited by an ellipsoid. We define $\mathcal{S}_d$ the set of initial data $u_0 \in \mathcal{H},$ which satisfy the additional symmetry conditions:
$$ \mathcal{S}_d := \{ u_0 \in  \mathcal{H}(\Omega)   \backslash     \; u_0(x_1,..,-x_i,..,x_d)=-u_0(x_1,..,x_i,..,x_d), \;  i=1, \ldots ,d  \} . $$ 
By uniqueness for the Cauchy problem, the symmetry properties of $u_0 \in \mathcal{S}_d$ are conserved, that is, 
 $$u(t, x_1,..,-x_i,..,x_d)=-u(t,x_1,..,x_i,..,x_d), \;  i=1,\ldots,d.$$ 
\begin{theo}
\label{theo sym}
Assume $d\geq 2$, $p \geq 1+\frac{4}{d} $ and $\Theta$ is invariant under the above symmetry.
Let $u_0 \in \mathcal{S}_d$ and let $u$ be the corresponding solution of \Nls with maximal time interval $I$ of existence. If $E[u]<0 $ and $|x|u_0 \in L^2(\Omega),$ then the length of $I$ is finite and thus the solution $u$ blows up in finite time.
\end{theo}

\begin{rem}
Theorems \ref{theo outside spehre} and \ref{theo sym} remain true for the \NNls equation outside an obstacle centered at any point $x_0 .$ For Theorem \ref{theo sym}, one would have to use a symmetry around $x_0$ instead of the origin. Moreover, we conjecture that the Theorems are still valid with weaker assumption $p>1+\frac 4d.$ \\ 
%We can also generalized these theorems for any dimension $d\geq 4,$ whenever an appropriate well-posedness theory is available. In dimension $d \geq 4$, for Theorem \ref{theo sym} we should suppose $d$~symmetries,  
%$$u_0(x_1,..,-x_i,..,x_d)=-u_0(x_1,..,x_i,..,x_d), \quad \text{ for }\; i=1,2,...,d.$$
%More details are given in sections \ref{sectionBlow-upAnyDimension} and \ref{sectionBlow-upSymmetricSolution2D}.
\end{rem}
Now we introduce the concept of ground state. Let $Q$ be the solution of the following nonlinear elliptic equation 
\begin{equation}
\label{ellip-eq}
    -Q+\Delta Q + |Q|^{p-1}Q=0, \quad Q=Q(x), \quad x \in \R^d.
\end{equation}

For $1<p<\frac{d+2}{d-2}$, this nonlinear equation has infinite number of solutions in $H^1(\R^d)$. Among these there is exactly one solution which is positive and radial, called the ground state solution. It is the unique minimal mass solution up to space translation and phase shift and exponentially decaying, see~\cite{Kwong89}. We henceforth denote by $Q$ this ground state solution.
\begin{theo}
\label{theo threshold on d=2}
Assume that $\Theta$ is invariant by the symmetry defined above and $ s_c= \frac{d}{2} - \frac{2}{p-1}.$ 
Let $u_0$ be such that 
\begin{itemize}
    \item for $d=2,\, u_0\in \mathcal{S}_2$
   and $s_c > 0,$ i.e, $p > 3$. %
    \item for $d\geq3,\,u_0 \in \mathcal{S}_d $ and $0 < s_c <1, $ i.e, $1 + \frac 4d < p <\frac{d+2}{d-2}$. 
\end{itemize}
and let $u$ be the corresponding solution of \Nls with maximal time interval of existence $I.$ Suppose that
 \begin{equation}
\label{Euu<EQ d=2}
   M[u_0]^{\frac{1-s_c}{s_c}} E[u_0] <  M[Q]^{\frac{1-s_c}{s_c}} E[Q].
\end{equation}
If \eqref{Euu<EQ d=2} holds and \begin{equation}
\label{Duu0>DQ d=2}
 \left\| u_0 \right\|_{L^2(\Omega)}^{1-s_c}   \left\| \nabla u_0 \right\|_{L^2(\Omega)}^{s_c} > \left\| Q  \right\|_{L^2{\footnotesize{\left(\R^2 \right)}}}^{1-s_c} \left\| \nabla Q \right\|_{L^2{\footnotesize{\left(\R^2 \right)}}}^{s_c}.
\end{equation}
Then for $t\in I $, 
\begin{equation}
\label{Duu(t)>D(Q) d=2}
 \left\| u_0 \right\|_{L^2(\Omega)}^{1-s_c}   \left\| \nabla u(t) \right\|_{L^2(\Omega) }^{s_c}> \left\| Q  \right\|_{L^2\footnotesize{\left(\R^2 \right)}}^{1-s_c} \left\| \nabla Q \right\|_{L^2\footnotesize{\left(\R^2 \right)}}^{s_c}.
\end{equation}
Furthermore, if $|x|u_0 \in L^2(\Omega)$ then the length of $I$ is finite and thus the solution blows up in finite time.

\end{theo}

Let us mention that in the $L^2$-critical case we can find an almost explicit blow-up solution for the \NNls equation using pseudo-conformal transformation. In this case, we can construct 
a blow-up solution for the \NNls equation by adapting the argument of N.\,Burq, P.G{\'e}rard and N.\,Tzvetkov in \cite{BuGeTz03}, for the \nnls equation inside a domain. \\ 

Indeed, assume $p=1+\frac 4d.$ Let $\Psi $ be a $\mathcal{C}^{\infty}$-function such that  $\Psi=0$ near $\Theta$ and  $\Psi=1$ for $|x|>>1$ and let $Q$ be any solution of the nonlinear elliptic equation \eqref{ellip-eq}, (it does not have to be the ground state) then there exists $ T>0 $ and a smooth function $r(t,x)$ defined on $[0,T)\times \Omega$ and exponentially decaying as $t \to T$ such that  
$$u (t,x):=\frac{1}{(T-t)} Q\left(\frac{x-x_0}{T-t}\right) \Psi(x) e^{i(\frac{4-|x-x_0|^2}{4(T-t)})}+r(t,x) $$
 is solution for \Nls satisfying the Dirichlet boundary conditions, which blow-up in finite time~$T$. The proof is similar to one given in \cite{BuGeTz03} for \nls equation inside a domain in $\R^2.$ We need to construct the smooth correction $r$ such that $u$ is a solution of \Nls satisfying Dirichlet boundary conditions. To achieve this, one define a contraction mapping using the Duhamel formula on a closed ball in the Banach space $(E, \left\| \cdot \right\|_E)$ defined by   $$E:=\{f\in \mathcal{C}([0,T),H^2(\Omega) \cap H^1_0(\Omega)); \left\| f \right\|_{E} < \infty  \},$$ equipped with the norm $$ \left\|f\right\|_E:=\sup_{t\in [0,T)} \{  e^{\frac{1}{2(T-t)}} \left\| f \right\|_{L^2(\Omega)}+ e^{ \frac{1}{3(T-t)}} \left\| f \right\|_{H^2(\Omega)} \}. $$ 
The existence of the smooth correction $r$ follows from the fixed point theorem.\\
%on the complete metric space $(E,d)$ defined by

The paper is organized as follows. In section \ref{sectionPropertiesoftheGroundState}, we give a review of some properties related to the ground state $Q$. In section \ref{sectionPohozaev&MorawetzIdentities}, we prove Pohozaev's identities outside an obstacle. In section~\ref{sectionBlow-upAnyDimension}, we prove the existence of blow-up solution in the exterior of a ball for $p\geq 5$ and outside a convex obstacle that satisfies \eqref{obs-type} for $p\geq 1+\frac{4}{d-\frac{M}{m}(d-1)} $, using a convexity argument on the modified variance. In section \ref{sectionBlow-upSymmetricSolution2D}, we study the existence of symmetric blow-up solution for $p \geq 1+ \frac 4d$ using a different variance. Finally, in section \ref{sectionThresholdBlow-upSolution}, we study the behavior of the solutions, in particular, the blow-up criteria for the solutions with initial data beyond the ground state threshold. %The criterion is expressed in terms of the scale-invariant quantities $\left\|u_0\right\|_{L^2(\Omega)}\left\| \nabla u_0 \right\|_{L^2(\Omega)}$ and $M(u)E(u).$ 

\section*{Acknowledgements}
O.L would like to thank Thomas Duyckaerts (LAGA) and Svetlana Roudenko (FIU) for their valuable comments and suggestions which helped improve the manuscript. Part of this work was done while the author was visiting Department of Mathematics and Statistics at Florida International University, Miami,US. 
\section{Properties of the Ground State}
\label{sectionPropertiesoftheGroundState}
Weinstein \cite{Weinstein82} proved that the sharp constant $C_{GN}$ in the Gagliardo-Nirenberg estimate 
\begin{equation}
\label{Nirem}
    \left\| f \right\|_{L^{p+1}}^{p+1} \leq C_{GN}   \left\| \nabla f \right\|_{L^2}^{\frac{d(p-1)}{2}} 
    \left\| f \right\|_{L^2}^{2-\frac{(d-2)(p-1)}{2}} 
\end{equation}

is attained at the function $Q$ (the ground state described in the introduction), i.e  
$$C_{GN}= \frac{\left\| Q\right\|_{L^{p+1} \footnotesize{\left(\R^d \right)}}^{p+1}} 
{ \left\| \nabla Q \right\|_{L^2 \footnotesize{\left(\R^d \right)}}^{ \frac{d(p-1)}{2}} \left\| Q \right\|_{L^2 \footnotesize{ \left(\R^d \right)}}^{2-\frac{(d-2)(p-1)}{2}} } \; .$$
Multiplying \eqref{ellip-eq} by $Q$ and integrating by parts, we obtain
\begin{equation}
\label{ 1 iden de Q}
    - \left\| Q \right\|_{L^2\footnotesize{\left(\R^d \right)}}^2 - \left\| \nabla Q \right\|_{L^2 \footnotesize{\left(\R^d \right)}} + \left\| Q \right\|_{L^{p+1}\footnotesize{\left(\R^d \right)}}^{p+1}=0.
\end{equation}
Multiplying \eqref{ellip-eq} by $x.\nabla Q$ and integrating by parts, we obtain the following identity 
\begin{equation}
\label{ 2 iden de Q}
\frac{d}{2} \left\| Q \right\|_{L^2\footnotesize{\left(\R^d \right)}}^2 + \frac{d-2}{2} \left\| \nabla Q \right\|_{L^2\footnotesize{\left(\R^d \right)}}^2 - \frac{d}{p+1} \left\| Q \right\|_{L^{p+1}\footnotesize{\left(\R^d \right)}}^{p+1} =0.
\end{equation}
These two identities \eqref{ 1 iden de Q} and  \eqref{ 2 iden de Q} enable us to obtain these relations
\begin{align*}
    \left\| \nabla Q \right\|_{L^2\footnotesize{\left(\R^d \right)}}^2 &= \frac{d(p-1)}{(d+2)-p(d-2)} \left\| Q \right\|_{L^2\footnotesize{\left(\R^d \right)}}^2   
    \\ \left\| Q \right\|_{L^{p+1}\footnotesize{\left(\R^d \right)}}^{p+1} &=\frac{2(p+1)}{d(p-1)} \left\| \nabla Q \right\|_{L^2\footnotesize{\left(\R^d \right)}}^2 
\end{align*}
and thus, reexpress 
\begin{equation}
\label{c_GN}
C_{GN}= 
\left( \frac{2(p+1)}{d(p-1)}  \left\| \nabla Q \right\|_{L^2\footnotesize{\left(\R^d \right)}}
\left\| Q \right\|_{L^2\footnotesize{\left(\R^d \right)}}^{\frac{4-(d-2)(p-1)}{d(p-1)-4}}\right)^{-\frac{d(p-1)-4}{2}}
\end{equation}
We also compute
\begin{equation}
\label{energyQ}
    E[Q]:=\frac{1}{2} \left\| \nabla Q \right\|_{L^2\footnotesize{\left(\R^d \right)}}^2- \frac{1}{p+1} \left\| Q \right\|_{L^{p+1}\footnotesize{\left(\R^d \right)}}^{p+1}= \frac{d(p-1)-4}{2d(p-1)} \left\| \nabla Q \right\|_{L^2\footnotesize{\left(\R^d \right)}}^2.
\end{equation}
\section{Pohozaev's identities outside obstacle}
\label{sectionPohozaev&MorawetzIdentities}
This section is devoted to the proof of the Pohozaev's Identity in exterior domain. In the following Proposition $\Omega$ can be the exterior of any regular obstacle. 
%Recall that, $\Omega=\R^d \backslash \Theta,$ where $\Theta=B(0,1)$ and let $u$ be a solution of \Nls with Dirichlet boundary condition, i.e $u=0 $ on $\partial \Omega.$ 
%In the following computations, we assume $d \in \{2,3\}.$ However, the computations remain valid for $d\geq 4 $ whenever there exists an appropriate Cauchy theory for \Nls equation in Sobolev spaces.
\begin{prop}[Pohozaev's identity]
\label{poho inde}
Let $u \in H^2 \cap H^1_0(\Omega), |x|\, u \in L^2(\Omega)$ then we have, 
\begin{equation}
\label{phoz 1}
 \re \intOm  \Delta \overline{u} \;  ( \,  \frac{d}{2} \, u + x.\nabla u ) \,  dx  = - \intOm \left| \nabla u \right|^2 dx 
 + \frac{1}{2}  \intb   \left| \nabla u \right|^2 \; (x.\nn ) \,  d\sigma(x) .
\end{equation}
\begin{multline}
\label{phoz 2}
  \re \displaystyle  \intOm \Delta \overline{u} \; \left(  \nabla u.\frac{x}{|x|} + (\frac{d-1}{2} )  \frac{u }{|x|}    \right) \,dx =  -\frac{ (d-1)(d-3)}{4}\intOm \frac{|u|^2}{|x|^3} \,  dx  - \intOm \frac{ \left| \nablaT u \right|^2}{|x|} \, dx
  \\ +  \frac{1}{2} \intb \left| \nabla u \right|^2 \frac{x. \nn }{ |x|} \, d\sigma(x),
\end{multline}
  where $\left| \nablaT u \right|^2 :=\left| \nabla u \right|^2- \left|\frac{x}{|x|}. \nabla u \right|^2 $ and $\nn$ is the outward unit normal vector.
\end{prop}
\begin{proof}
Using integration by parts and the fact that $u$ satisfies Dirichlet boundary condition (i.e $u=0$ on $\partial \Omega$), we obtain
\begin{align*}
    \re \intOm \Delta \overline{u} \, (x.\nabla u ) \, dx   &= - \sum_{j=1}^d \re \intOm  \partial_{x_j}
 \overline{u} \, \partial_{x_j} u  \, dx -\sum_{j,k=1}^d \re \intOm   \, x_k  \, \partial_{x_j}\partial_{x_k} u \,  \partial_{x_j} \overline{u} \, dx 
 \\ &+  \intb   \left| \nabla u \right|^2 \; (x.\nn ) \,  d\sigma(x) 
 \\ & = - \intOm | \nabla u |^2 \, dx + \frac{1}{2} \sum_{k=1}^d \intOm | \nabla u |^2 \, dx - \frac{1}{2} \sum_{k=1}^d \intb |\nabla u |^2  \, x_k n_k \, d\sigma(x)
     \\ &+  \intb   \left| \nabla u \right|^2 \; (x.\nn ) \,  d\sigma(x) 
\\ &= - \intOm | \nabla u |^2 \, dx - \frac{d}{2}  \re \intOm \Delta \overline{u} \, u \, dx  + \frac{1}{2 } \intb   \left| \nabla u \right|^2 \; (x.\nn ) \,  d\sigma(x).
\end{align*}
This conclude the proof of \eqref{phoz 1}. Now let us prove \eqref{phoz 2}, using the same argument as above and the fact that,
\begin{equation*}
\partial_{x_k} \left( \frac{x_j}{|x|} \right)= \begin{cases}
 \frac{1}{|x|}-\frac{x_j^2}{|x|^3}\;  \quad \text{ if } j=k ,  \\ \\  -\frac{x_j \, x_k}{|x|^3}  \quad  \quad \text{ if } j \neq k ,
\end{cases}
\end{equation*}
we obtain 
\begin{align*}
    \re \displaystyle  \intOm \Delta \overline{u} \;   \nabla u.\frac{x}{|x|} \, dx &= -\frac{1}{2}  \left[ \sum_{j,k=1}^d \intOm \partial_{x_k} \overline{u} \; \partial_{x_k}   \partial_{x_j} u \; \frac{x_j}{|x|} + \partial_{x_k}u \; \partial_{x_k} \partial_{x_j} \overline{u} \; \frac{x_j}{|x| } \, dx   \right]
    \\ &- \re \left[ \sum\limits_{\substack{j=1 }}^{d} \intOm \partial_{x_j}\overline{u} \;  \partial_{x_j} u \left( \frac{1}{|x|}-\frac{x_j^2}{|x|^3} \right) \, dx  \right] + \re \left[ \sum\limits_{\substack{j,k=1 \\ j \neq k}}^{d}  \intOm \partial_{x_k} \overline{u} \; \partial_{x_j} u \;  \frac{x_j\, x_k }{|x|^3} \,  dx   \right]
    \\ &+  \intb \left| \nabla u \right|^2 \frac{(x.\nn) }{|x|} d\sigma(x) 
  \end{align*}

\begin{align*}
\re \displaystyle  \intOm \Delta \overline{u} \;   \nabla u.\frac{x}{|x|} \, dx    &=   -\frac{1}{2} \left[ \sum_{j=1}^d \intOm \frac{x_j}{|x|} \partial_{x_j} \left( \left| \nabla u  \right|^2  \right) \, dx  \right]-\intOm \frac{\left| \nabla u \right|^2 }{|x|} \, dx +  \intOm  \left| \frac{x}{|x|}. \nabla  u  \right|^2 \, \frac{1}{|x|} \, dx     
 \\ &+  \intb \left| \nabla u \right|^2 \frac{(x.\nn)}{|x|} d\sigma(x)  \\
 &= \left(\frac{d-1}{2}  \right)\intOm \frac{\left| \nabla u  \right|^2 }{|x|} \, dx - \intOm \frac{\left| \nablaT u \right|^2 }{|x|} \, dx + \frac{1}{2}\intb \left| \nabla u \right|^2 \frac{(x.\nn)}{|x|} d\sigma(x) 
\\ &= \left( \frac{d-1}{2} \right) \re \intOm - \Delta \overline{u} \; u \, \frac{1}{|x|} \, dx + \left( \frac{d-1}{2} \right)  \re \sum_{k=1}^d \intOm  \partial_{x_k} \overline{u} \, u  \; \frac{x_k}{|x|^3} \, dx 
\\ & - \intOm \frac{\left| \nablaT u \right|^2 }{|x|} \, dx + \frac{1}{2}\intb \left| \nabla u \right|^2 \frac{(x.\nn)}{|x|} d\sigma(x)  .
\end{align*}
Using the fact that, $$ \displaystyle \left(\frac{d-1}{2} \right) \re \sum_{k=1}^d \intOm \partial_{x_k}\overline{u} \, u \, \frac{x_k}{|x|^3} \, dx = -\frac{ (d-1)(d-3)}{4}\intOm \frac{|u|^2}{|x|^3} \,  dx,$$ \\ 
we obtain  \eqref{phoz 2},
%\begin{align*}
  % \re \displaystyle  \intOm \Delta \overline{u} \; \left(  \nabla % u.\frac{x}{|x|} + (\frac{d-1}{2} )  \frac{u }{|x|}    \right) \,dx &= - \frac{(d-1)(d-3)}{4} \intOm \frac{|u|^2}{|x|^3} \,  dx %  - \intOm \frac{ \left| \nablaT u \right|^2}{|x|} \, dx
  %\\ &+  \frac{1}{2} \intb \left| \nabla u \right|^2 \frac{x. \nn %}{ |x|} \, d\sigma(x).
%\end{align*}
which concludes the proof of Proposition \ref{poho inde}.
\end{proof}

\section{Existence of blow-up solution}
\label{sectionBlow-upAnyDimension}
This section is devoted to the proofs of Theorem \ref{theo outside spehre} and \ref{theo-convex}. We assume $d\in \{2,3 \}. $ Nevertheless, the computations below still valid for $d \geq 4$ if an appropriate Cauchy theory is available. \\ 

Denote: \begin{equation*}
\Upsilon_1(u(t)):=\intOm |x| \left|u(t,x)\right|^2 dx, \qquad 
    \Upsilon_2(u(t)):= \intOm  \left| x \right|^2 \left| u(t,x) \right|^2 dx.
\end{equation*}

We will start by proving the following virial identities in the exterior of a convex obstacle, in particular in the exterior of a ball which is needed in the proof of Theorem \ref{theo outside spehre}.
%We will prove the following identities 
%We will use the following virial identities in exterior domain in the proof of Theorem \ref{theo outside spehre}. 
\subsection{Virial identities in exterior domain}

\begin{prop}
\label{Propo 1&2 deri of Up1 & Up2}
Assume that $\Theta$ is a smooth, compact and convex obstacle. \\ 
Let $u_0 \in H^2 \cap H^1_0(\Omega), |x|\, u_0 \in L^2(\Omega)$ and let $u$ be the corresponding solution of the \NNls equation.  Then 
%\vspace{-.5cm}
\begin{align}
\label{1 deri of |.|^2}
\frac{d}{dt}\Upsilon_2(u(t))&= 4 \; \im \intOm \overline{u}(t,x) \; x.\nabla u(t,x) dx.%= 4\;  W(u(t)).
\\ \label{2 deri of |.|^2}
 \frac{1}{16} \frac{d^2}{dt^2} \Upsilon_2(u(t))&= E [u]- \frac{1}{2} \left( \frac{d}{2}  - \frac{d+2}{p+1} \right)\intOm |u|^{p+1} \, dx - \frac{1}{4 } \intb \left| \nabla u \right|^2 \, (x.\nn) \, d\sigma(x) .
  \end{align}
  And 
\begin{equation}
\label{1 deri of |.|}
\hspace*{-1.65cm} \frac{d}{dt}\Upsilon_1(u(t))= 2 \; \im \sum_{j=1}^d \intOm \overline{u}(t,x) \; \frac{x_j}{|x|}.\partial_{x_j} u(t,x) dx. %= 2\;  \WW(u(t)).  
\end{equation}
\begin{multline}
\label{2 deri of |.|}
  \frac{1}{16}  \frac{d^2}{dt^2} \Upsilon_1(u(t)) =  \frac{(d-1)(d-3)}{16}  \intOm \frac{|u|^2}{|x|^3} \,  dx    -\frac{(d-1)(p-1)}{8(p+1)}\intOm \frac{|u|^{p+1}}{|x|} \, dx \\ 
  +\frac{1}{4}  \intOm \frac{ \left| \nablaT u \right|^2}{|x|} \, dx
 -\frac{1}{8}  \intb \left| \nabla u \right|^2 \frac{x. \nn }{ |x|} \, d\sigma(x).
\end{multline}
where $\left| \nablaT u \right|^2 :=\left| \nabla u \right|^2- \left|\frac{x}{|x|}. \nabla u \right|^2 $ and $\nn$ is the outward unit normal vector.
\end{prop}

\begin{proof}
Multiplying the equation by $|x|^2 \overline{u}$ and taking the imaginary part we get,
$$\im \intOm i \partial_t u |x|^2 \overline{u } dx + \im \intOm \Delta u |x|^2 \overline{u } dx = -\im \intOm |u|^{p-1} u |x|^2 \overline{u} dx=0. $$
Which yields,  $$ \displaystyle \frac{1}{2}\frac{d}{dt} \Upsilon_2 (u(t))=   \frac{1}{2} \frac{d}{dt} \intOm |x|^2 \left|u(t,x)\right|^2  dx = - \im \intOm |x|^2 \Delta u \; \overline{u} \, dx. $$ \\ 
Integration by parts ensures $$\frac{1}{2} \frac{d}{dt} \intOm |x|^2 |u(t,x)|^2 dx = 2 \sum_{k=1}^d \im \intOm x_k \partial_{x_k} u \, \overline{u}  dx = 2 \, \im \intOm \overline{u } \;  x.\nabla u \, dx.$$
This implies \eqref{1 deri of |.|^2}. 
Now let us compute the second derivative of $\Upsilon_2$. 
\begin{align*}
\frac{d^2}{dt^2} \Upsilon_2(u(t)):&= 4 \frac{d}{dt } \im \intOm \overline{u} \, x.\nabla u \, dx \\ 
& = 4 \left( \im \intOm \partial_t \overline{u} \;  x.\nabla u dx + \im \intOm \overline{u} \;  x.\nabla(\partial_t u ) dx    \right)  
\\ & = 4 \left( \im \intOm (-i\Delta \overline{u} -i |u|^{p-1}\overline{u} ) \; x.\nabla u \; dx + \im \intOm \overline{u}\; 
x.\nabla(i\Delta u + i |u|^{p-1} u )dx 
\right) \\ &= 4 \bigg[ \underbrace{ \re \intOm - \Delta \overline{u } \;  x.\nabla u  \,  dx }_{\rm I_1}
+ \underbrace{ \re \intOm \overline{u} \, x.\nabla (\Delta u) \, dx }_{\rm I_2}  + \underbrace{\re \intOm -|u|^{p-1} \overline{u}\; x.\nabla u \, dx }_{\rm I_3}  \\ & + \underbrace{ \re \intOm \overline{u} \, x.\nabla (|u|^{p-1} u ) \; dx }_{\rm I_4} \bigg].
\end{align*}
Next, we compute each integral apart, using integration by parts and the Dirichlet boundary condition, i.e., $u=0 $ on $\partial {\Omega}.$
\begin{align*}
    \rm{I_2}:&= \re \sum_{k=1}^d \intOm \overline{u} \; x_k \partial_{x_k} \Delta u dx 
    = \sum_{k=1}^d \re \intOm - \partial_{x_k} (\overline{u}  x_k) \Delta u \, dx + \re \sum_{k=1}^{d} \intb \overline{u} \, \Delta u \;  (x_k n_k)\,  d\sigma(x) 
\\ &= \re \intOm -\nabla \overline{u}.x \, \Delta u \, dx - d \re \intOm \overline{u} \, \Delta u \, dx .
 \end{align*}
 Hence, 
 \begin{equation*}
    {\rm{ I_1+I_2}}:= -2 \re \intOm \Delta \overline{u} \; ( x.\nabla u + \frac{d}{2} \, u ) \, dx .
 \end{equation*}
 
Using  Pohozaev's Identity \eqref{phoz 1}, we get 
\begin{align*}
   \rm{I_1+I_2}:=& 2 \intOm \left| \nabla u \right|^2 dx   - \intb \left| \nabla u \right|^2 \; (x.\nn ) \,  d\sigma(x). 
\end{align*}
\begin{equation*}
   { \rm{I_4}}:=  \re \intOm \overline{u} \, x.\nabla (|u|^{p-1} u ) \; dx = - \re \intOm \nabla \overline{u}.x \; |u|^{p-1}\, u \, dx - d \intOm |u|^{p+1} \, dx.
\end{equation*}
Using the fact that
\begin{equation}
\label{nabla |u|^p+1}
\nabla ( |u|^{p+1}) = (p+1)\, |u|^{p-1} \; \re \left(\overline{u} \, \nabla u \right),  
\end{equation} 
we obtain
\begin{align*}
    \rm{I_3 + I_4} &=  -2 \re \intOm |u|^{p-1} \overline{u}\; x.\nabla u \, dx - d \intOm  |u|^{p+1} \, dx  
    \\ &= -\frac{2}{p+1} \re \intOm x.\nabla( |u|^{p+1} ) \, dx - d \intOm |u|^{p+1} \, dx 
    \\ &= \left( \frac{2 \,d }{p+1}  - d \right)
    \intOm |u|^{p+1} \, dx.
\end{align*}
Which yields 
\begin{equation*}
    \frac{d^2}{dt^2} \Upsilon_2(u(t)) = 8 \intOm | \nabla u |^2 \, dx 
 + \left(\frac{8 \, d }{p+1} - 4d \right) \intOm |u|^{p+1} \, dx - 4 \intb |\nabla u |^{2} \, (x.\nn) \, d\sigma(x).
 \end{equation*}
Thus
\begin{equation*}
    \frac{1}{16}  \frac{d^2}{dt^2} \Upsilon_2(u(t))= E [u]- \frac{1}{2} \left( \frac{d}{2}  - \frac{d+2}{p+1} \right) \intOm |u|^{p+1} \, dx - \frac{1}{4} \intb \left| \nabla u \right|^2 \, (x.\nn) \, d\sigma(x).
\end{equation*}
This concludes the proof of \eqref{2 deri of |.|^2}. Now let us compute the first derivative of $\Upsilon_1$. Similarly, multiplying the equation by $|x| \overline{u}$ and taking the imaginary part we get,
\begin{equation*}
 \frac{1}{2}   \frac{d}{dt} \Upsilon_1(u(t))= \frac{d}{dt} \intOm |x| |u(t,x)|^2 dx= \im \intOm -\Delta u \, |x|\, \overline{u} \,  dx 
    = \im  \sum_{j=1}^d \intOm \frac{x_j}{|x|} \partial_{x_j} u \; \overline{u} \, dx . 
\end{equation*}
Thus, we obtain \eqref{1 deri of |.|}
\begin{equation*}
\frac{d}{dt}\Upsilon_1(u(t))= 2 \; \im \sum_{j=1}^d \intOm \overline{u}(t,x) \; \frac{x_j}{|x|}.\partial_{x_j} u(t,x) dx = 2\;  \WW(u(t)).  
\end{equation*}
For the second derivative of $\Upsilon_1$, using the \Nls we get 
\begin{align*}
\frac{d^2}{dt^2} \Upsilon_1(u(t))&= 2 \im  \sum_{j=1}^d \intOm \partial_{t} \overline{u} \frac{x_j}{|x|}  \partial_{x_j}u \, dx   +   2 \im  \sum_{j=1}^d \intOm \overline{u} \, \frac{x_j}{|x|} \partial_{x_j} (\partial_{t} u) \, dx      
\\&= \underbrace{2 \re \sum_{j=1}^d \intOm -\Delta \overline{u} \, \frac{x_j}{|x|} \, \partial_{x_j} u \, dx }_{\rm{J_1}}+ \underbrace{2 \re  \sum_{j=1}^d \intOm \overline{u}\, \frac{x_j}{|x|} \, \partial_{x_j}(\Delta u ) \, dx  }_{\rm{J_2}} 
\\ &+ \underbrace{2 \re \sum_{j=1}^d  \intOm -|u|^{p-1} \, \overline{u} \, \frac{x_j}{|x|} \, \partial_{x_j}u \, dx   }_{\rm{J_3}}+ \underbrace{2 \re  \sum_{j=1}^d \intOm \overline{u}   \, \frac{x_j}{|x|} \, \partial_{x_j} ( |u|^{p-1}u) \, dx     }_{\rm{J_4}} .
\end{align*}

We will compute each integral apart, using again integration by parts and the Dirichlet boundary condition, i.e., $u=0 $ on $\partial {\Omega}.$
\begin{align*}
    \rm{J_2}:&= 2 \re  \sum_{j=1}^d  \intOm \overline{u} \, \frac{x_j}{|x|} \, \partial_{x_j}(\Delta u ) \, dx  = 2 \re  \sum_{j=1}^d \intOm - \partial_{x_j} \left( \overline{u} \frac{x_j}{|x|} \right) \Delta u \, dx  
 \\ &= -2  \re   \intOm \nabla \overline{u}.\frac{x}{|x|} \, \Delta u \,  dx  -2  \re  \intOm  \overline{u}  \, \frac{(d-1)}{|x|} \, \Delta u \, dx .
\end{align*}
Hence, 
\begin{align*}
    \rm{J_1+J_2}:&= -4  \re   \intOm  \Delta \overline{u} \, \nabla u.\frac{x}{|x|} \, \,  dx  -2  \re  \intOm   \Delta \overline{u}   \, \frac{(d-1)}{|x|} \, u \, dx 
    \\ &= -4 \left[ \re \intOm \Delta \overline{u} \left( \nabla u . \frac{x}{|x|} + \frac{(d-1)}{2} \, u\, \frac{1}{|x|} \, \right) \, dx     \right] .
\end{align*}

Using  \eqref{phoz 2},  we get 
\begin{align*}
    \rm{J_1+J_2}&= -4 \left[ -\frac{ (d-1)(d-3)}{4}\intOm \frac{|u|^2}{|x|^3} \,  dx   - \intOm \frac{ \left| \nablaT u \right|^2}{|x|} \, dx
  +  \frac{1}{2} \intb \left| \nabla u \right|^2 \frac{x. \nn }{ |x|} \, d\sigma(x)  \right]
  \\ &=   (d-1)(d-3)\intOm \frac{|u|^2}{|x|^3} \,  dx   +4  \intOm \frac{ \left| \nablaT u \right|^2}{|x|} \, dx-  2 \intb \left| \nabla u \right|^2 \frac{x. \nn }{ |x|} \, d\sigma(x).   
\end{align*}
\begin{align*}
{ \rm{J_4}}:&= 2 \re  \sum_{j=1}^d  \intOm \overline{u} \, \frac{x_j}{|x|} \, \partial_{x_j}( |u|^{p-1} u ) \,  dx    
    \\& = -2 \re  \sum_{j=1}^d \intOm \partial_{x_j} \overline{u} \, \frac{x_j}{|x|} |u|^{p-1} u \, dx     
    -2 \re  \sum_{j=1}^d  \intOm \overline{u} \;  \partial_{x_j}\left( \frac{x_j}{|x|}\right) |u|^{p-1} u \, dx    
    \\ &= -2 \re \intOm \nabla \overline{u}. \frac{x}{|x|} \, |u|^{p-1} u \, dx -2(d-1)\intOm \frac{|u|^{p+1}}{|x|} \, dx. 
 \end{align*}
Due to \eqref{nabla |u|^p+1}, we have 
 \begin{align*}
     \rm{J_3+J_4}:&= -4 \re \intOm \nabla \overline{u}. \frac{x}{|x|} |u|^{p-1} u \, dx - 2 (d-1) \intOm \frac{|u|^{p+1}}{|x| } \, dx
     \\ &= \frac{-4}{p+1}  \intOm \frac{x}{|x|}. \nabla( |u|^{p+1}) \, dx -2(d-1) \intOm \frac{|u|^{p+1}}{|x|} \, dx 
     \\ &= \frac{-2(d-1)(p-1)}{p+1} \intOm \frac{|u|^{p+1}}{|x|} \, dx . 
  \end{align*} 
Summing all terms, we get 
\begin{align*} 
\frac{d^2}{dt^2} \Upsilon_1(u(t))&= \frac{d^2}{dt^2} \intOm |x| \, |u|^2 \, dx 
\\&=     (d-1)(d-3) \intOm \frac{|u|^2}{|x|^3} \,  dx    +4  \intOm \frac{ \left| \nablaT u \right|^2}{|x|} \, dx - \frac{2(d-1)(p-1)}{p+1} \intOm \frac{|u|^{p+1}}{|x|} \, dx 
\\ &-  2 \intb \left| \nabla u \right|^2 \frac{x. \nn }{ |x|} \, d\sigma(x).
\end{align*}
This concludes the proof of Proposition \ref{Propo 1&2 deri of Up1 & Up2}.
\end{proof}
\subsection{Existence of blow-up solution in the exterior of a ball}
\begin{proof}[Proof of Theorem \ref{theo outside spehre}]
Assume $\Theta=B(0,R)$ and $p\geq 5.$ Let $u_0 \in \mathcal{H}(\Omega)$ (for $d=2,3,$ it suffies to take $u_0 \in H^2 \cap H^1_0(\Omega),$ we will later relax the assumption to $u_0 \in H^1_0(\Omega)$ if $d=2$), $|x|u_0 \in L^2(\Omega),$ $E[u_0]+ \frac{1}{8R^2} M[u_0]<0 $ if d=2 and $E[u]<0$ if $d \geq 3.$ %Recall that, 
Let $u$ be the corresponding solution of \Nls outside the ball $B(0,R)$, with maximal time interval $I$ of existence. Define the variance used in this proof: 
\begin{equation*}
    \mathcal{V}(u(t)):=\intOm \left( |x|^2-2R|x|+10 \right) |u(t,x)|^2 \, dx.
\end{equation*}
From Proposition  \ref{Propo 1&2 deri of Up1 & Up2} we have 
\begin{multline}
\label{2-der of total Virial}
\frac{1}{16}\frac{d^2}{dt^2} \mathcal{V}(u(t)) = E [u] -\frac{R}{2}  \intOm \frac{ \left| \nablaT u \right|^2}{|x|} \, dx - \frac{1}{2} \left( \frac{d}{2}  - \frac{d+2}{p+1} \right)  \intOm |u|^{p+1} \, dx \\ 
\qquad \; \,  +\frac{R(d-1)(p-1)}{4(p+1)} \intOm \frac{|u|^{p+1}}{|x|} \; dx -   \frac{ R(d-1)(d-3)}{8}   \intOm \frac{|u|^2}{|x|^3} \,  dx  \\ 
   - \frac{1}{4 } \intb \left| \nabla u \right|^2 \, (x.\nn) \, d\sigma(x)  + \frac{R}{4}  \intb \left| \nabla u \right|^2 \frac{x. \nn }{ |x|} \, d\sigma(x) .   \qquad \qquad \qquad \qquad \;
\end{multline}
 Let us control first the boundary terms.  Recall that, in Theorem \ref{theo outside spehre} we assume that $\Omega$ is the exterior of the ball $B(0,R).$ We denotes by $\nn$ the outward unit normal vector, i.e the normal vector exterior to $\Omega$, so that $|x|=R$ and $x.\nn \leq 0$ on $\partial \Omega=\partial B(0,R)$.

%Recall that $\Omega$ is the exterior of the ball, so that $x\cdot \nn \leq 0 $ and $|x|=R$ for $x\in \partial \Omega$. 
\begin{equation*}
- \frac{1}{4 } \intb \left| \nabla u \right|^2 \, (x.\nn) \, d\sigma(x)  + \frac{R}{4}  \intb \left| \nabla u \right|^2 \frac{x. \nn }{ |x|} \, d\sigma(x)   =\frac{1}{4} \intb \left| \nabla u \right|^2 (x. \nn ) \left(\frac{R}{|x|}-1\right) \, d\sigma(x) = 0 .
\end{equation*}

Now, we will estimate the nonlinear terms. Using the fact that $\frac{1}{|x|}\leq \frac 1R,$ for all $x\in \Omega,$ and $p \geq 5 ,$ we have
\begin{align*}
- \frac{1}{2} \left( \frac{d}{2}  - \frac{d+2}{p+1} \right)\intOm |u|^{p+1} \, dx & +  \, \frac{R(d-1)(p-1)}{4(p+1)}\intOm \frac{|u|^{p+1}}{|x|} \, dx 
\\ &\leq \left[ - \frac{1}{2} \left( \frac{d}{2}  - \frac{d+2}{p+1} \right)+\frac{(d-1)(p-1)}{4(p+1)}   \right] \intOm |u|^{p+1} \, dx 
\\ & = - \left(\frac{p-5}{4(p+1)} \right) \intOm |u|^{p+1} \, dx \leq 0 .
\end{align*}
Finally, for all $d \neq 2$ one can see that,
\begin{equation}
%\label{restrictionDimension}
\frac{- R(d-1)(d-3)}{8}  \intOm \frac{|u|^2}{|x|^3} \,  dx \leq 0. 
\end{equation}
In particular, for $d=3$ we have $\frac{- (d-1)(d-3)}{8}  \displaystyle \intOm \frac{|u|^2}{|x|^3} \,  dx=0$. \\ 
For $d=2,$ we use the fact that, $E[u]+\frac{1}{8 R^2}M[u] <0 $ and $\frac{1}{|x|} \leq \frac 1R$ for all $x\in \Omega.$ Indeed, 
$$ E[u]+\frac{R}{8} \intOm \frac{|u|^2}{|x|^3} \,  dx \leq E[u]+\frac{1}{8 R^2}M[u] <0 .  $$
This implies that the second derivative of the variance is bounded by a negative constant, for all $t\in I$. 
\begin{equation*}
\frac{d^2}{dt^2} \mathcal{V}(u(t)) \leq -A ,  \; \text{ where \;  -A=} \left\{ \begin{array}{rrrrrrrr}
\begin{aligned}
	 E[u]+\frac{1}{8 R^2}M[u]&<0 \; \text{ if } d=2,  \\ 
     E[u]&<0 \; \text{ if } d=3.
		\end{aligned}
			\end{array}
		\right.
\end{equation*}
Moreover, integrating twice over $t$, we have that 
\begin{equation}
\label{dervVV}
\mathcal{V}(u(t))\leq -A t^2+B t+C, \text { where  }B=\frac{d}{dt} \mathcal{V}(u_0) \text{ and } C=\mathcal{V}(u_0).
\end{equation}
By density \eqref{dervVV} remains true, if $d=2,$ assuming that $u_0 \in  H^1_0(\Omega)$ and $|x| u_0 \in L^2(\Omega).$ Due to \eqref{dervVV}, there exists $T^*$ such that $ \mathcal{V}(u(T^*))<0 ,$ which is a contradiction. Then the length of the maximal time interval of existence $I$ is finite and one can prove that the solution $u$ blows up in finite time. This concludes the proof of Theorem~\ref{theo outside spehre}.
\end{proof}

%Furthermore, the restrict in dimension $d\neq 2$ for negative energy, i.e $E[u]<0$, is coming from the following terms which does not have a favorable sign for our purpose, $$ -   \frac{ (d-1)(d-3)}{8}  \intOm \frac{|u|^2}{|x|^3} \,  dx \leq 0. $$ Thus we treated the case $d=2$ separately. For that, we have 
%$$E[u]-   \frac{1}{8}  \intOm \frac{|u|^2}{|x|^3} \,\leq E[u]+\frac{1}{8}M[u] <0.$$

\subsection{Existence of blow-up solution in the exterior of a convex obstacle}
\label{Existence-blow-up-Convex}
In this section, we extend the previous results in the exterior of a ball to a smooth, compact, convex obstacle $\Theta. $ We prove Theorem \ref{theo-convex}, assuming without loss of generality that $0 \in \Theta,$ and we suppose that $\Theta$ satisfies the following property for  $d\geq 2$:
\begin{equation}
\label{M/m}
\frac{M}{m} < \frac{d}{d-1}, \text{ where } M= \displaystyle  \max_{x \in \partial \Theta} \left(|x|\right) \text{ and }m= \displaystyle \min_{x\in \partial \Theta}\left( |x|  \right). 
\end{equation}
We use the following variance identity: 
\begin{equation*}
    \mathcal{V}(u(t)):=\intOm \left( |x|^2- 2 M\, |x|+10 \right) |u(t,x)|^2 \, dx.
\end{equation*}

\begin{proof}[Proof of Theorem \ref{theo-convex}]
Assume that $\Theta$ satisfies \eqref{M/m} and $p\geq 1+\frac{4}{d-\frac{M}{m}(d-1)}. $ Let  $u_0 \in \mathcal{H},$ $|x|u_0 \in L^2(\Omega),$ and suppose that
$E[u_0]+ \frac{M}{8m^3} M[u_0]<0 ,$ if $d=2,$ and $E[u_0]<0 ,$ if $ d\geq 3.$ Let $u$ be the corresponding solution of the \NNls equation in the exterior of a convex obstacle $\Theta,$ such that the assumption \eqref{M/m} holds, with maximal time interval $I$ of existence.  \\ 

From Proposition \ref{Propo 1&2 deri of Up1 & Up2}, we have 
\begin{multline}
\label{2-der of total Virial}
\frac{1}{16}\frac{d^2}{dt^2} \mathcal{V}(u(t)) = E [u] -\frac{M}{2}  \intOm \frac{ \left| \nablaT u \right|^2}{|x|} \, dx - \frac{1}{2} \left( \frac{d}{2}  - \frac{d+2}{p+1} \right)  \intOm |u|^{p+1} \, dx \\ 
\qquad \qquad \quad  \; \; \,  +\frac{M(d-1)(p-1)}{4(p+1)} \intOm \frac{|u|^{p+1}}{|x|} \; dx 
    -   \frac{ M (d-1)(d-3)}{8}  \intOm \frac{|u|^2}{|x|^3} \,  dx \\ 
     - \frac{1}{4 } \intb \left| \nabla u \right|^2 \, (x.\nn) \, d\sigma(x)  + \frac{M}{4}  \intb \left| \nabla u \right|^2 \frac{x. \nn }{ |x|} \, d\sigma(x) .   \qquad \qquad \qquad \; \; \;\,
\end{multline}
We first control the boundary terms.  Recall that $\Omega$ is the exterior of a convex obstacle $\Theta$  and $0 \in \Theta$, so that $x\cdot \nn \leq 0$ for all $x\in \partial \Omega.$ As $ M=\displaystyle \max_{x \in \partial \Theta}(|x|)=\max_{x \in \partial \Omega}  (|x|),$ then $(\frac{M}{|x|}-1)\geq 0$ for all $x \in \partial \Omega.$ Thus, 
\begin{equation*}
- \frac{1}{4 } \intb \left| \nabla u \right|^2 \, (x.\nn) \, d\sigma(x)  + \frac{M}{4}  \intb \left| \nabla u \right|^2 \frac{x. \nn }{ |x|} \, d\sigma(x)   = \frac 14 \intb \left| \nabla u \right|^2 (x. \nn )  \left( \frac{M}{|x|} -1\right)) d\sigma(x) \leq 0.
\end{equation*}
Next, we control the nonlinear terms using the fact that $\frac{M}{m} < \frac{d}{d-1},$ $  p \geq 1+ \frac{4}{d-\frac{M}{m}(d-1)}$ and $\frac{1}{|x|}\leq \frac 1m, $ for all $ x \in \Omega.$ %As  $m$ is the minimum of all $x$ in the boundary of $ \Omega$  then $\frac{1}{|x|}\leq \frac 1m, $ for all $ x \in \Omega.$ Thus, 
\begin{align*}
- \frac{1}{2} \left( \frac{d}{2}  - \frac{d+2}{p+1} \right)\intOm |u|^{p+1} \, dx & +  \, \frac{M(d-1)(p-1)}{4(p+1)}\intOm \frac{|u|^{p+1}}{|x|} \, dx 
\\ &\leq \left[ - \frac{1}{2} \left( \frac{d}{2}  - \frac{d+2}{p+1} \right)+\frac{M(d-1)(p-1)}{4m(p+1)}   \right] \intOm |u|^{p+1} \, dx 
\\ & = - \left(\frac{(p-1)\left(d-\frac{M}{m}(d-1)  \right)-4}{4(p+1)} \right) \intOm |u|^{p+1} \, dx \leq 0 .
\end{align*}
For all $d \neq 2,$ one can see that all other terms are negative. For $d=2,$ we use the fact that, $\frac{1}{|x|} \leq \frac 1m$ for all $x\in \Omega$ and $E[u]+\frac{M}{8 m^3}M[u] <0 ,$ to obtain 
$$ E[u]+\frac{M}{8} \intOm \frac{|u|^2}{|x|^3} \,  dx \leq E[u]+\frac{M}{8 m^3}M[u] <0 .  $$
This implies that the second derivative of the variance is bounded by a negative constant, for all $t\in I$. 
\begin{equation*}
\frac{d^2}{dt^2} \mathcal{V}(u(t)) \leq -\mathcal{A} ,  \; \text{ where }\;  -\mathcal{A}= \left\{ \begin{array}{rrrrrrrr}
\begin{aligned}
	 E[u]+\frac{M}{8m^3}M[u]&<0 \; \text{ if } d=2,  \\ 
     E[u]&<0 \; \text{ if } d \geq 3.
		\end{aligned}
			\end{array}
		\right.
\end{equation*}

Using the same argument as in the proof of Theorem \ref{theo outside spehre}, one can prove that the solution~$u$ blows up in finite time and this concludes the proof of Theorem \ref{theo-convex}

\end{proof}

\section{Existence of blow-up symmetric solution}
\label{sectionBlow-upSymmetricSolution2D}
%\subsection{Existence of blow-up solution in the exterior of a sphere}
In this section we prove Theorem \ref{theo sym}. For the sake of simplicity, we will give the proof for $d=2,$ then we will generalize the result to any dimension $d\geq 3.$ \\

 Assume that $d = 2, \; p\geq 3 $ and $\Theta$ is invariant by the symmetry, $x_j\longmapsto -x_j ,$ Recall that define $\mathcal{S}_2$ is defined as the following 
 $$ \mathcal{S}_2 := \{ u_0 \in \mathcal{H}(\Omega)   \backslash    u_0(-x_1,x_2)=u_0(x_1,-x_2)=-u_0(x_1,x_2) \} ,$$ 
here, we can consider $\mathcal{H}(\Omega)=H^1_0(\Omega).$ By uniqueness for the Cauchy problem, we have  $$u(t,-x_1,x_2)=u(t,x_1,-x_2)=-u(t,x_1,x_2). $$ 

Due to the two symmetry assumptions, we will reduce the problem to a quarter of the space. We define 
\begin{align*}
 \Omega:&=\bigcup \Omega^{\pm \pm}:=  \Omega^{+} \, \cup  \, \Omega^{-}\quad  
 \text{ and } \quad \Omega^{\pm \pm}:=\{ (x_1,x_2)  \in \Omega \backslash \, x_1 \in \R^\pm \text{ and }  x_2 \in \R^\pm  \},  \\ 
 \;  \Omega^{\pm}:&=\{(x_1,x_2)\in \Omega \backslash \, x_1 \in \R^\pm  \} \quad \text{ and } \quad \Omega_{\pm}:=\{(x_1,x_2)  \in \Omega \backslash \, x_2 \in \R^{\pm}\}.
\end{align*}
Furthermore, since $u_{{\vert_{x_1=0}}}=0$ and $u_{\vert_{x_2=0}}=0,$ one can see that $u$ satisfies the Dirichlet boundary conditions on each set defined above, $\Omega^{\pm\pm}$ and $ \Omega^{\pm}.$ \\

%
%\begin{rem}
%Let us mention that, due to the continuity of the flow, the symmetry properties of $u_0$, i.e are conserved, that is,  $$u(t,-x_1,x_2)=u(t,x_1,-x_2)=-u(t,x_1,x_2). $$ Furthermore, one can see that $u$ satisfies the Dirichlet boundary conditions on each set defined above, $\Omega^{\pm\pm}, \Omega^{\pm}.$ 
%\end{rem}
  
The variance identity here is the following:  
Let $C>0$ be a positive constant to be specified later,
$${\rm{V}}(u(t)):= \intOm \left( |x|^2-C |x_1|-C |x_2|+C^2 \right) \left| u(t,x) \right|^2 \, dx  .$$
Denote 
\begin{equation*}
\Gamma_1(u(t)):= \intOm |x_1| |u(t,x)|^2 dx , \qquad  \qquad \; \,  \Gamma_2(u(t)):= \intOm  |x_2| |u(t,x)|^2 dx.
\end{equation*}

\begin{prop}
\label{deriv-dt-1&2-Gamma 1&2}
Let $u_0 \in H^2 \cap H^1_0(\Omega)$ and $|x|\, u_0 \in L^2(\Omega)$ such that $u_0 \in \mathcal{S}_2$ and let $u$ be the corresponding solution of the \NNls equation. Then 
\begin{align}
\label{dt gamma 1}
    \frac{d}{dt} \Gamma_1(u(t))&= 8\,  \im \int_{\Omega^{++}} \partial_{x_1} u(t,x) \, \overline{u}(t,x) \, dx, \\
\label{dt^2 gamma 1}
    \frac{d^2}{dt^2} \Gamma_1(u(t))&= 8 \int_{\partial \Omega^{++}} \left| \nabla u(t,x) \right|^2 |n_1| \, dx,
\end{align}
and 
\begin{align}
\label{dt gamma 2}
\frac{d}{dt} \Gamma_2(u(t))&= 8 \, \im \int_{\Omega^{++}} \partial_{x_2} u (t,x) \,\overline{u}(t,x) dx, \\
\label{dt^2 gamma 2}
 \frac{d^2}{dt^2} \Gamma_2(u(t)) &= 8 \int_{\partial \Omega^{++}} \left| \nabla u(t,x) \right|^2 |n_2| \, dx.
\end{align}
\end{prop}

\begin{proof}
Multiply the equation by $|x_1| \bar{u}$ and take the imaginary part to get: $$ \im \intOm i \partial_t u |x_1| \overline{u} \, dx + \im \intOm \Delta u |x_1| \overline{u}  \, dx= -\underbrace{ \im \intOm |u|^{p-1}u |x_1| \overline{u} \, dx}_{=0} , $$
which yields 
\begin{align*}
    \frac{1}{2} \frac{d}{dt} \intOm |x_1| |u(t,x)|^2 \, dx &= - \im \intOm |x_1| \Delta u \, \overline{u} \, dx \\ 
    \frac{d}{dt} \Gamma_1(u(t)):&= \frac{d}{dt} \intOm |x_1| |u(t,x)|^2 dx= -2  \im \intOm |x_1| \, \Delta u \, \overline{u} \, dx .
\end{align*}
Integration by parts ensures,
\begin{align*}
   \im \intOm |x_1| \Delta u \, \overline{u} \, dx &= -  \im  \int_{\Omega} \partial_{x_1} ( |x_1|) \,\partial_{x_1} u \,  \overline{u}\, dx   \\
  & = - \left(\im \int_{\Omega^{+}} \partial_{x_1}u \,  \overline{u} \, dx - \im \int_{\Omega^{-}}  \partial_{x_1}u    \, \overline{u} \, dx    \right) \\ 
  &= -4 \;  \im \int_{\Omega^{++}} \partial_{x_1}u \, \overline{u} \, dx,
\end{align*}
which yields \eqref{dt gamma 1}. Now, let us compute the second derivative of $\Gamma_1(u(t)).$\\ 
Denote: \begin{equation*}
 \alpha(t):= \im \int_{\Omega^{++}} \overline{u}(t,x) \; x.\nabla u(t,x) dx , \quad  \beta(t):=  \im \int_{\Omega^{++}} \overline{u}(t,x) \,  (x-e_1).\nabla u(t,x) dx. 
\end{equation*}
Thus, we have $$ \frac{d}{dt}\Gamma_1(u(t))=8(\alpha(t)-\beta(t)).$$ 
Due to the symmetry properties of $u$, one can see that $\alpha(t)$ is equal to $\frac{1}{16}\frac{d}{dt} \Upsilon_2(t)$ and $\beta(t)$ is equal to $\eqref{1 deri of |.|^2}$ applied to $(x-e_1)$, where $e_1=(1,0)$.
By \eqref{2 deri of |.|^2}, we obtain, \begin{align*}
\frac{d}{dt}\alpha(t)&= 4 E_+[u]
 - 2 \left( \frac{d}{2}  - \frac{d+2}{p+1} \right)\int_{\Omega^{++}} |u|^{p+1} \, dx 
 -   \int_{\partial{\Omega^{++}}} \left| \nabla u \right|^2 \, (x.\nn) \, d\sigma(x) , \\ \frac{d}{dt}\beta(t)&= 4 E_+[u]
  - 2 \left( \frac{d}{2}  - \frac{d+2}{p+1} \right)\int_{\Omega^{++}} |u|^{p+1} \, dx 
    -  \int_{\partial{\Omega^{++}}} \left| \nabla u \right|^2 \, ((x-e_1).\nn) \, d\sigma(x) .
\end{align*}
where, $\displaystyle E_{+}[u]=\frac{1}{2}\left\| \nabla u \right\|_{L^2(\Omega^{++})}- \frac{1}{p+1} \int_{\Omega^{++}}  |u|^{p+1} dx.   $ Taking the difference between these two equalities, we get
\begin{align*}
\frac{d}{dt} \left(  \alpha(t)-\beta(t) \right)=- \int_{\partial \Omega^{++}} \left| \nabla u \right|^{2} (e_1.\nn) \, dx= -  \int_{\partial \Omega^{++}} \left| \nabla u \right|^{2} n_1 \, dx,
\end{align*}
which yields \begin{equation*}
    \frac{d^2}{dt^2} \Gamma_1(u(t)):= -8  \int_{\partial \Omega^{++}} \left| \nabla u \right|^2 n_1 \, dx .
\end{equation*}
Using the symmetry properties, the convexity of the obstacle $\Theta$ and the fact that $\nn$ is the unit outward normal vector, we have $n_1 \leq 0$ on $\partial \Omega^{++}.$
Then, we obtain
 \begin{equation*}
 \frac{d^2}{dt^2} \Gamma_1(u(t))= 8 \int_{\partial \Omega^{++}} \left| \nabla u \right|^2 |n_1| \, dx.
 \end{equation*}
Recall that  $$\Gamma_2:= \intOm  |x_2| |u(t,x)|^2 dx.$$
We multiply the equation by $|x_2| \bar{u}$ and take the imaginary part to get: $$ \im \intOm i \partial_t u |x_2| \overline{u} \, dx + \im \intOm \Delta u |x_2| \overline{u}  \, dx= - \im \intOm |u|^{p-1}u |x_2| \overline{u} \, dx =0. $$
Which yields 
\begin{align*}
    \frac{1}{2} \frac{d}{dt} \intOm |x_2| |u(t,x)|^2 \, dx &= - \im \intOm |x_2| \Delta u \, \overline{u} \, dx \\ 
    \frac{d}{dt} \Gamma_2(u(t)):&= \frac{d}{dt} \intOm |x_2| |u(t,x)|^2 dx= -2  \im \intOm |x_2| \, \Delta u \, \overline{u} \, dx .
\end{align*}
Integration by parts ensures, 
\begin{align*}
   \im \intOm |x_2| \Delta u \, \overline{u} \, dx &= -  \im  \int_{\Omega} \partial_{x_2} ( |x_2|) \,  \partial_{x_2}u \, \overline{u}\, dx   \\
  & = - \left(\im \int_{\Omega_{+}} \partial_{x_2}u \,  \overline{u} \, dx - \im \int_{\Omega_{-}}  \partial_{x_2}u    \, \overline{u} \, dx    \right) 
\\  &= -4 \im \int_{\Omega^{++}} \partial_{x_2}u \,  \overline{u} \, dx,
\end{align*}
which yields \eqref{dt gamma 2}.
The computation of $\frac{d^2}{dt^2}\Gamma_2(t)$ is similar to the one of $\frac{d^2}{dt^2}\Gamma_1(t),$ the only difference is to apply \eqref{1 deri of |.|^2} and \eqref{2 deri of |.|^2} to $(x-e_2),$ where $e_2=(0,1)$ instead of $(x-e_1).$ Thus, we get
\begin{equation*}
    \frac{d^2}{dt^2} \Gamma_2(u(t)):= - 8 \int_{\partial \Omega^{++}} \left| \nabla u \right|^2 n_2 \, dx , 
\end{equation*}
As $\nn$ is the unit outward normal vector, we have $n_2 \leq 0$ on $ \partial \Omega^{++}$ using the symmetry properties and the convexity of the obstacle $\Theta.$ Then we obtain,
\begin{equation*}
    \frac{d^2}{dt^2} \Gamma_2(u(t)):= 8 \int_{\partial \Omega^{++}} \left| \nabla u \right|^2 |n_2| \, dx  .
\end{equation*}
This concludes the proof of Proposition \ref{deriv-dt-1&2-Gamma 1&2}.
\end{proof}
\begin{proof}[Proof of Theorem \ref{theo sym}]
Assume $d=2$ and $ p\geq 3$. Let $u_0 \in \mathcal{S}_2$ (i.e., $u_0 \in H^1_0(\Omega), \; u_0(-x_1,x_2)=u_0(x_1,-x_2)=-u_0(x_1,x_2)),$ $|x|u_0 \in L^2(\Omega)$ and $E[u_0]<0$ and llet $u$ be the corresponding solution of \Nls with maximal time interval $I$ of existence. From Proposition \ref{deriv-dt-1&2-Gamma 1&2} and \ref{Propo 1&2 deri of Up1 & Up2}, we deduce the second derivative of the variance for $d=2,$
\begin{multline}
\label{dt^2 V on 2D}
 \frac{d^2}{dt^2} {\rm{V}}(u(t))=16 E [u]- 8 \left( \frac{p-3}{p+1} \right)\intOm |u|^{p+1} \, dx \\
 +  \int_{\partial \Omega^{++}} \left| \nabla u \right|^2 \, \big[16 |x.\nn| -8C \left(|n_1| + |n_2| \right) \big] \, d\sigma(x) .
\end{multline}

Using the fact that $p \geq 3$ and $E[u]<0 $, one can see that the first two terms are negative, i.e., $16 E[u]-8\left(\frac{p-3}{p+1}\right)\intOm |u|^{p+1}dx \leq 0.$
Choosing $C\geq 2 \displaystyle\max_{x\in \partial \Omega}\left( \frac{|x.\nn|}{|n_1|+|n_2|}\right)$ this implies that,
\begin{equation*}
\frac{d^2}{dt^2} {\rm{V}}(u(t)) \leq -A, \text{ \; where \;} A>0. 
\end{equation*}

Using the same argument as in the first proof, one can prove that the length of the maximal time interval of existence $I$ is finite. Therefore, the solution $u$ blows up in finite time and this concludes the proof of Theorem~\ref{theo sym} in dimension $2.$ For any dimension $d\geq 3$, we should suppose that, $u_0\in \mathcal{S}_d$, i.e., 
$u_0(x_1,..,x_i,..,x_d)=-u_0(x_1,..,-x_i,..,x_d),$ for $i=1,2,...,d$ and we use the following variance: \begin{equation*}
    {\rm{V}}(u(t))=\intOm \left( |x|^2-C\sum_{i=1}^d |x_i|+C^2 \right) \left| u(t,x) \right|^2 \, dx. 
\end{equation*} 
One can check that 
\begin{multline}
\label{dt^2 V on any D}
 \frac{d^2}{dt^2} {\rm{V}}(u(t))=16 E [u]- 8 \left( \frac{d}{2}-\frac{d+2}{p+1} \right)\intOm |u|^{p+1} \, dx \\
 +  \int_{ \partial\{ x_i \geq 0 ,  \, 1 \leq i \leq d \} } \left| \nabla u \right|^2 \, \big[2^{d+2} |x.\nn| -2^{d+1} C \sum_{i=1}^d |n_i|\big] \, d\sigma(x) .
\end{multline}
Using the fact that $p \geq 1+\frac{4}{d},$ $E[u]<0 $ and choosing $  C$ such that $$ C \geq 2 \displaystyle \max_{x\in \partial \Omega}  \left( |x.\nn|  \left(\displaystyle \sum_{i=1}^d|n_i|\right)^{-1}\right),$$ we get 
\begin{equation*}
\frac{d^2}{dt^2} {\rm{V}}(u(t)) \leq -A, \text{ \; where \;} A>0. 
\end{equation*}
Then $u $ blows up in finite time for any dimension $d\geq 3$ and this concludes the proof of Theorem~\ref{theo sym}.  
\end{proof}

\section{Ground state threshold for blow-up}
\label{sectionThresholdBlow-upSolution}
This section is devoted to the proof of Theorem \ref{theo threshold on d=2}. First, assume that $d=2$ and $s_c>0,$ i.e.,  $p> 3$ and  $\Theta$ is invariant by the symmetry. Let $u_0 \in \mathcal{S}_2,\,  (u_0\in   H^1_0(\Omega),$ $u_0(-x_1,x_2)=u_0(x_1,-x_2)=-u_0( x_1,x_2)),$ and $|x|\, u_0 \in L^2(\Omega)$. Let $u$ be the corresponding solution of the \NNls equation with maximal time interval $I$ of existence. Moreover, we consider the same modified variance as in the proof of Theorem \ref{theo sym}.  Let $C>0$ be a positive constant to be specified later,

$${\rm{V}}(u(t)):= \intOm \left(|x|^2-C|x_1|-C|x_2|+ C^2 \right) \left| u(t,x) \right|^2 \, dx $$ 
\begin{lem}
Let $u_0 \in H^1_0(\Omega)$ satisfy 
\begin{align}
\label{EM}
    M[u_0]^{1-s_C} E[u_0]^{s_c} &<   M[Q]^{1-s_c} E[Q]^{s_c}. \\
    \label{MQ}
    \left\| u_0 \right\|_{L^2(\Omega)}^{1-s_c} \left\| \nabla u_0 \right\|_{L^2(\Omega)}^{s_c} &> \left\| Q \right\|_{L^2\footnotesize{\left(\R^2 \right)}}^{1-s_c} \left\| \nabla Q \right\|_{L^2\footnotesize{\left(\R^2 \right)}}^{s_c}.
\end{align}
Then the corresponding solution $u$ to \Nls satisfy, 
\begin{equation}
\label{MQ(t)}
  \forall t \in I ,\qquad   \left\| u_0 \right\|_{L^2(\Omega)}^{1-s_c} \left\| \nabla u(t)\right\|_{L^2(\Omega)}^{s_c} > \left\| Q \right\|_{L^2\footnotesize{\left(\R^2 \right)}}^{1-s_c} \left\| \nabla Q \right\|_{L^2\footnotesize{\left(\R^2 \right)}}^{s_c}.
\end{equation}
 \end{lem}
 \begin{proof}
 The proof of the lemma is the same as in \cite{HoRo07},  \cite{HoRo08} for the proof of blow-up solutions of \nls equation on $\R^{d}.$ We give it for the sake of completeness and for the convenience of the reader. The key point is that a function $f \in H^1_0(\Omega)$ extended by $0$ outside $\Omega$ can be identified to an element of $H^1(\R^2). $ Thus, it satisfies the same Gagliardo-Nirenberg inequality as an element of $H^1(\R^2).$  \\ 
 
 Multiplying the energy by $M[u]^{\frac{1-s_c}{s_c}}$ and applying Gagliardo-Nirenberg's inequality for~$d=~2,$ we have,
 \begin{align*}
 M[u]^{\frac{1-s_c}{s_c}}E[u]&=\frac{1}{2} \left\| \nabla u \right\|^2_{L^2(\Omega)} \left\|u\right\|^{2\frac{(1-s_c)}{s_c}}_{L^2(\Omega)}-
 \frac{1}{p+1} \left\| u \right\|^{p+1}_{L^{p+1}(\Omega)} \left\|u\right\|^{2\frac{(1-s_c)}{s_c}}_{L^2(\Omega)} 
 \\ & \geq \frac{1}{2}\left( \left\| \nabla u \right\|_{L^2(\Omega)} \left\| u \right\|^{\frac{1-s_c}{s_c}}_{L^2(\Omega)} \right)^{2}-\frac{C_{GN}}{p+1}\left( \left\| \nabla u \right\|_{L^2(\Omega)} \left\| u \right\|^{\frac{1-s_c}{s_c}}_{L^2(\Omega)} \right)^{p-1}
 \\ &  \geq f \left(\left\| \nabla u \right\|_{L^2(\Omega)} \left\| u \right\|^{\frac{1-s_c}{s_c}}_{L^2(\Omega)} \right),
 \end{align*}
 where $f(x)=\frac{1}{2} x^2 - \frac{C_{GN}}{p+1} \,  x^{p-1} $. Then, $f'(x)= x-\frac{C_{GN}(p-1)}{p+1} x^{p-2}$, and thus, 
 $f'(x)=0$ for $x_0=0$ and $x_1=\left( \frac{C_{GN}(p-1)}{p+1} \right)^{-\frac{1}{p-3}}=\left\| \nabla Q \right\|_{L^2\footnotesize{\left(\R^2 \right)}} \left\| Q \right\|^{\frac{1-s_c}{s_c}}_{L^2\footnotesize{\left(\R^2 \right)}}$  by \eqref{c_GN}. Since \eqref{Nirem} is attained at ground state $Q$ then we have, $f(\left\| \nabla Q \right\|_{L^2\footnotesize{\left(\R^2 \right)}} \left\| Q \right\|_{L^2\footnotesize{\left(\R^2 \right)}}^{\frac{1-s_c}{s_c}} )= M[Q]^{\frac{1-s_c}{s_c}} E[Q], $ we also have $f(0)=0.$ Thus, the function $f$ is increasing on $(0,x_1)$ and decreasing on $(x_1,\infty)$. Using the energy conservation, we get 
 \begin{equation}
     \label{f}
f \left(\left\| \nabla u \right\|_{L^2(\Omega)} \left\| u \right\|^{\frac{1-s_c}{s_c}}_{L^2(\Omega)} \right) \leq M[u]^{\frac{1-s_c}{s_c}}E[u(t)]<f(x_1) .
\end{equation}
 
If condition \eqref{MQ} holds, i.e $\left\| u_0 \right\|_{L^2(\Omega)}^{1-s_c} \left\| \nabla u_0 \right\|_{L^2(\Omega)}^{s_c} > x_1=\left\| Q \right\|_{L^2\footnotesize{\left(\R^2 \right)}}^{1-s_c} \left\| \nabla Q \right\|_{L^2\footnotesize{\left(\R^2 \right)}}^{s_c}$, then by~\eqref{f} and the continuity of $\left\| \nabla u(t) \right\|_{L^2(\Omega)}$ in time we obtain \eqref{MQ(t)} for all time $t \in I $. 
 \end{proof}
 Moreover, if the conditions \eqref{EM} and \eqref{MQ} holds, then there exists $\delta_1 >0 $ such that 
\begin{equation}
\label{ref EQ}
    M[u_0]^{1-s_c} E[u_0 ]^{s_c} <  (1-\delta_1) M[Q]^{1-s_c} E[Q]^{s_c}. 
\end{equation}
Thus, there exists $\delta_2:=\delta_2(\delta_1) >0 $ such that  
\begin{equation}
\label{ref MQ}
 \forall t \in I,  \qquad   \left\| u_0 \right\|_{L^2(\Omega)}^{1-s_c} \left\| \nabla u(t)\right\|_{L^2(\Omega)}^{s_c} > (1+\delta_2) \left\| Q \right\|_{L^2\footnotesize{\left(\R^2 \right)}}^{1-s_c} \left\| \nabla Q \right\|_{L^2\footnotesize{\left(\R^2 \right)}}^{s_c} .
\end{equation}
Now let us prove that
\begin{multline}
\label{dt^2 rm V threshold}
\frac{d^2}{dt^2} {\rm{V}}(u(t)) \leq  8(p-1) E[u] -4(p-3)  \left\| \nabla u  \right\|_{L^2(\Omega)}^2 \\  +  \int_{\partial \Omega^{++}} \left| \nabla u \right|^2 \, \big[16 |x.\nn| -8C \left(|n_1| + |n_2| \right) \big] \, d\sigma(x).
\end{multline}
From \eqref{dt^2 V on 2D}, we have 
\begin{align*}
\frac{d^2}{dt^2} {\rm{V}}(u(t))&=16 E [u]- 8 \left(  \frac{p-3}{p+1} \right)\intOm |u|^{p+1} \, dx 
 +  \int_{\partial \Omega^{++}} \left| \nabla u \right|^2 \, \big[16 |x.\nn| -8C \left(|n_1| + |n_2| \right) \big] \, d\sigma(x) 
 \\ & \leq 8 \left\| \nabla u \right\|_{L^2}^2- \frac{8(p-1)}{p+1} \intOm |u|^{p+1} dx +  \int_{\partial \Omega^{++}} \left| \nabla u \right|^2 \, \big[16 |x.\nn| -8C \left(|n_1| + |n_2| \right) \big] \, d\sigma(x)
\\ & \leq   8(p-1) E[u] -4(p-3)  \left\| \nabla u  \right\|_{L^2(\Omega)}^2+  \int_{\partial \Omega^{++}} \left| \nabla u \right|^2 \, \big[16 |x.\nn| -8C \left(|n_1| + |n_2| \right) \big] \, d\sigma(x) .
\end{align*}

Multiplying \eqref{dt^2 rm V threshold} by $M[u]^{\frac{1-sc}{s_c} }$ and using \eqref{energyQ} for $d=2$ with the two refined inequalities \eqref{ref EQ} and \eqref{ref MQ}, we have 
\begin{align*}
M[u]^{\frac{1-sc}{s_c} }\frac{d^2}{dt^2} {\rm{V}}(u(t)) & \leq  \left( 8(p-1) E[u] -4(p-3)  \left\| \nabla u  \right\|_{L^2(\Omega)}^2   \right) M[u]^{\frac{1-sc}{s_c} }  
\\& +  \int_{\partial \Omega^{++}} \left| \nabla u \right|^2 \, \big[16 |x.\nn| -8C \left(|n_1| + |n_2| \right) \big] \, d\sigma(x)    M[u]^{\frac{1-sc}{s_c} }
\\& \leq 8(p-1) ( 1-\delta_1) E[Q] M[Q]^{\frac{1-sc}{s_c} }-4(p-3) (1+\delta_2) \left\| \nabla Q \right\|_{L^2\footnotesize{\left(\R^2 \right)}}^2  M[Q]^{\frac{1-sc}{s_c} }\\
&+  \int_{\partial \Omega^{++}} \left| \nabla u \right|^2 \, \big[16 |x.\nn| -8C \left(|n_1| + |n_2| \right) \big] \, d\sigma(x)    M[u]^{\frac{1-sc}{s_c} }
\\ &  \leq 8(p-1) ( 1-\delta_1) \frac{(p-3)}{2(p-1)} \left\| \nabla Q \right\|_{L^2\footnotesize{\left(\R^2 \right)}}^2 M[Q]^{\frac{1-sc}{s_c} }
\\ &-4(p-3) (1+\delta_2) \left\| \nabla Q \right\|_{L^2\footnotesize{\left(\R^2 \right)}}^2  M[Q]^{\frac{1-sc}{s_c} }
\\ & +  \int_{\partial \Omega^{++}} \left| \nabla u \right|^2 \, \big[16 |x.\nn| -8C \left(|n_1| + |n_2| \right) \big] \, d\sigma(x)    M[u]^{\frac{1-sc}{s_c} },
\end{align*}
which yields
\begin{align*}
M[u]^{\frac{1-sc}{s_c} }\frac{d^2}{dt^2} {\rm{V}}(u(t)) &  \leq \big[ 4(p-3)  - 4(p-3) \big]\left\| \nabla Q \right\|_{L^2\footnotesize{\left(\R^2 \right)}}^2 M[Q]^{\frac{1-sc}{s_c} }
\\ & -\big[ 4(p-3) \delta_1+ 4(p-3) \delta_2 \big] \left\| \nabla Q \right\|_{L^2\footnotesize{\left(\R^2 \right)}}^2 M[Q]^{\frac{1-sc}{s_c} }
\\ & +  \int_{\partial \Omega^{++}} \left| \nabla u \right|^2 \, \big[16 |x.\nn| -8C \left(|n_1| + |n_2| \right) \big] \, d\sigma(x)    M[u]^{\frac{1-sc}{s_c} }.
\end{align*}
Using the fact that $p > 3$ and choosing $C \geq 2 \displaystyle\max_{x\in \partial \Omega}\left( \frac{|x.\nn|}{|n_1|+|n_2|}\right)$ imply that the second derivative of the variance is bounded by a negative constant, for all $t\in I$, 
\begin{equation*}
    \frac{d^2}{dt^2} {\rm{V}}(u(t)) \leq -A, \;  \text{ where} \; A>0.
\end{equation*}
Thus the maximal time interval of existence $I$ is finite and the solution $u$ blows up in finite time. This concludes the proof of Theorem~\ref{theo threshold on d=2} in dimension $2$. \\  %The proof for dimension $3$ is similar to $d=2$. \\

Next, we will give the proof for 
dimension $d\geq 3$. For that, we suppose that  $u_0 \in \mathcal{S}_d,$ i.e.,
$$u_0(x_1,..,-x_i,..,x_d)=-u_0(x_1,..,x_i,..,x_d), \quad \text{ for }\; i=1,2,...,d.$$ 
Using the following variance \begin{equation*}
    {\rm{V}}(u(t))=\intOm \left( |x|^2-C\sum_{i=1}^d |x_i|+C^2 \right) \left| u(t,x) \right|^2 \, dx,
\end{equation*} 
we have, 
\begin{multline}
\label{dt^2 rm d=3 V threshold}
 \frac{d^2}{dt^2} {\rm{V}}(u(t))=4d(p-1)E[u]-(2d(p-1)-8)\left\|\nabla u\right\|_{L^2(\Omega)}^2  \\ +  \int_{ \partial\{ x_i \geq 0 ,  \, 1 \leq i \leq d \} } \left| \nabla u \right|^2 \, \bigg[2^{d+2} |x.\nn| -2^{d+1} C \sum_{i=1}^d |n_i|\bigg] \, d\sigma(x).
\end{multline}
Using the same argument as above, one can check that Lemma \ref{EM} remains true for $d\geq 3, $ see \cite{HoRo07}, \cite{Guevara14}.  If the conditions \eqref{EM} and \eqref{MQ} hold then there exists $\delta_1>0, \,\delta_2(\delta_1) >0 $ such that \eqref{ref EQ} and \eqref{ref MQ} are valid for $d\geq3$. Multiplying \eqref{dt^2 rm d=3 V threshold}  by $M[u]^{\frac{1-sc}{s_c} }$ and using  \eqref{energyQ} with the two refined inequalities  \eqref{ref EQ} and \eqref{ref MQ}, we have 
\begin{align*}
M[u]^{\frac{1-sc}{s_c} }\frac{d^2}{dt^2} {\rm{V}}(u(t)) & \leq 4d(p-1) ( 1-\delta_1) E[Q] M[Q]^{\frac{1-sc}{s_c}}\\
&-(2d(p-1)-8) (1+\delta_2) \left\| \nabla Q \right\|_{L^2\footnotesize{\left(\R^d \right)}}^2  M[Q]^{\frac{1-sc}{s_c} }\\
& +  \int_{ \partial\{ x_i \geq 0 ,  \, 1 \leq i \leq d \} } \left| \nabla u \right|^2 \, \bigg[2^{d+2} |x.\nn| -2^{d+1} C \sum_{i=1}^d |n_i|\bigg] \, d\sigma(x)    M[u]^{\frac{1-sc}{s_c} }
\\ &\leq (2d(p-1)-8) ( 1-\delta_1)  \left\| \nabla Q \right\|_{L^2\footnotesize{\left(\R^d \right)}}^2 M[Q]^{\frac{1-sc}{s_c} }
\\ &-(2d(p-1)-8) (1+\delta_2) \left\| \nabla Q \right\|_{L^2\footnotesize{\left(\R^d \right)}}^2  M[Q]^{\frac{1-sc}{s_c} }
\\ & +  \int_{ \partial\{ x_i \geq 0 ,  \, 1 \leq i \leq d \} } \left| \nabla u \right|^2 \, \bigg[2^{d+2} |x.\nn| -2^{d+1} C \sum_{i=1}^d |n_i|\bigg] \, d\sigma(x)    M[u]^{\frac{1-sc}{s_c} },
\end{align*}
which yields
\begin{align*}
 M[u]^{\frac{1-sc}{s_c} }\frac{d^2}{dt^2} {\rm{V}}(u(t)) &\leq   -\big[ (2d(p-1)-8) \delta_1+ (2d(p-1)-8) \delta_2 \big] \left\| \nabla Q \right\|_{L^2\footnotesize{\left(\R^d \right)}}^2 M[Q]^{\frac{1-sc}{s_c} }
\\ & +  \int_{  \partial\{ x_i \geq 0 ,  \, 1 \leq i \leq d \}  } \left| \nabla u \right|^2 \, \bigg[2^{d+2} |x.\nn| -2^{d+1}\, C \sum_{i=1}^d |n_i|\bigg] \, d\sigma(x) \,  M[u]^{\frac{1-sc}{s_c} }.
\end{align*}
Thus $$\displaystyle \frac{d^2}{dt^2} {\rm{V}}(u(t)) \leq -A, \text{ \; where \;} A>0. $$ Provided $p > 1+\frac{4}{d}$ and $ C \geq 2 \displaystyle \max_{x\in \partial \Omega}  \left( |x.\nn|  \left(\displaystyle \sum_{i=1}^d|n_i|\right)^{-1}\right).$  Then the solution $u$ blows up in finite time and this concludes the proof of Theorem \ref{theo threshold on d=2}.

%\pause
%\begin{align*}
%\textcolor{red}{\mathcal{S}_2}&:= \left\{u_0 \in H^1_0(\Omega) \backslash \quad {\color{red}{  u_0(-x_1,x_2)=u_0(x_1,-x_2)=-u_0( x_1,x_2)  }}  \right\} . \\ 
%%\color{red}{ \mathcal{S}_{3} }&:= \{ u_0 \in H^2 \cap H^1_0(\Omega)   \backslash \\  &  \color{red}{ u_0(-x_1,x_2,x_3)=u_0(x_1,-x_2,x_3)=u_0( x_1,x_2,-x_3)=-u_0(x_1,x_2,x_3)  } \} .   
%\color{red}{ \mathcal{S}_{d} }&:= \{ u_0 \in  H^1_0(\Omega)   \backslash \\  & { \color{red}{ \qquad u_0(x_1,..,-x_i,..,x_d)=-u_0(x_1,..,x_i,..,x_d), \;  i=1,2,...,d }} \} .  
%

\bibliographystyle{acm}
\bibliography{main.bbl}
\end{document}